\documentclass[10pt]{amsart}

\usepackage{graphicx}

\usepackage{color}

\usepackage[USenglish]{babel}
\usepackage{latexsym}
\usepackage{amsmath}
\usepackage{amsfonts}
\usepackage{amssymb}
\usepackage{esint}
\usepackage[all]{xy}
\renewcommand{\a }{\alpha }
\renewcommand{\b }{\beta }
\renewcommand{\d}{\delta }
\newcommand{\D }{\Delta }

\newcommand{\M}{\mathcal{M}}

\newcommand{\e }{\varepsilon }
\newcommand{\g }{\gamma}

\newcommand{\G }{\Gamma }
\renewcommand{\l }{\lambda }
\renewcommand{\L }{\Lambda }

\newcommand{\n }{\nabla }

\newcommand{\s }{\sigma }

\renewcommand{\O }{\Omega }

\newcommand{\ov}{\overline}

\newcommand{\wtilde }{\widetilde}

\newcommand{\be}{\begin{equation}}
\newcommand{\ee}{\end{equation}}

\newcommand{\R}{\mathbb{R}}
\renewcommand{\S}{\mathbb{S}}

\newcommand{\Z}{\mathbb{Z}}

\newcommand{\N}{\mathbb{N}}

\newcommand\ber{\begin{eqnarray}}\newcommand\bea{\begin{eqnarray}}
\newcommand\eer{\end{eqnarray}}\newcommand\eea{\end{eqnarray}}
\newcommand\berr{\begin{eqnarray*}}
\newcommand\eerr{\end{eqnarray*}}
\newcommand\ba{\begin{array}}
\newcommand\ea{\end{array}}

\newcommand\re{\mathrm{e}}
\newcommand\ri{\mathrm{i}}
\newcommand{\ud}{\mathrm{d}}
\newcommand{\nm}{\nonumber}
\newcommand{\ds}{\displaystyle}

\newtheorem{theorem}{Theorem}[section]
\newtheorem{proposition}[theorem]{Proposition}

\newtheorem{example}[theorem]{Example}

\newcommand{\bpr}{\begin{proposition}}
\newcommand{\epr}{\end{proposition}}
\newcommand{\bex}{\begin{example}\rm}
\newcommand{\eex}{\end{example}}

\begin{document}

\newtheorem{lem}{Lemma}[section]
\newtheorem{pro}[lem]{Proposition}
\newtheorem{thm}[lem]{Theorem}
\newtheorem{rem}[lem]{Remark}
\newtheorem{cor}[lem]{Corollary}
\newtheorem{df}[lem]{Definition}

\title[Non-relativistic Chern-Simons theory]{Existence results for a non-relativistic Chern-Simons model with \\ purely mutual interaction}

\author{Aleks Jevnikar}
\address[Aleks Jevnikar]{Department of Mathematics, Computer Science and Physics,
    University of Udine, Via delle Scienze 206, 33100 Udine, Italy}
\email{aleks.jevnikar@uniud.it}

\author{Sang-Hyuck Moon}
\address[Sang-Hyuck Moon] {Department of Mathematics and Institute of Mathematical Sciences, Pusan National University, Busan 46241, Republic of Korea}
\email{shmoon@pusan.ac.kr}

\begin{abstract}
We are concerned with a skew-symmetric singular Liouville system arising in non-relativistic Chern-Simons theory.
Based on its variational structure, we establish existence and multiplicity results.
Since the energy functional is indefinite, standard variational approaches do not apply directly.
We overcome this difficulty by introducing a suitable constrained problem and implementing a Morse-theoretical argument.
\end{abstract}

\thanks{The authors are grateful to X. Han and G. Tarantello for discussions on the first version of the paper.}


\maketitle


\

\section{Introduction}

\medskip

In recent decades, abelian Chern–Simons gauge theories in $(2+1)$-dimensions have been used to model various planar condensed-matter phenomena, including anyonic excitations and the fractional quantum Hall effect.
In the non-relativistic setting, self-dual vortices arise as finite-energy static configurations of complex scalar fields coupled to Chern–Simons gauge fields. 
In the present paper we are concerned with a class of such equations arising from a non-relativistic $[U(1)]^2$ Chern–Simons model describing two interacting species of matter fields in the purely mutual
interaction regime.

In multi-component Chern–Simons theories, the interaction between different species is encoded in a coupling matrix $K$, and the associated self-dual equations give rise to nonlinear elliptic systems of Liouville type with interaction matrix $K$. A prototypical example is the work of Kim, Lee, Ko, Lee and Min \cite{klkl}, where several Schr\"odinger fields are coupled to $[U(1)]^N$ abelian Chern–Simons gauge fields. For a suitable choice of self-dual potential, the static Bogomol'nyi equations reduce to a coupled system of the form
\[
  \Delta \log |\Psi_p|^2
  = -\sum_{p'} K_{pp'} |\Psi_{p'}|^2 + \cdots,
\]
so that self-dual multi-anyon configurations are described by a generalized Liouville system with interaction matrix $K$. Generalized Liouville systems with a non-negative symmetric coupling matrix have been extensively studied from the PDE point of view; see, for instance, \cite{lz,lz2,LinZhangCritical} and the references therein.

In this paper we focus on the following system:
\be \label{system0}
-\Delta u_i
= \lambda \sum_{j=1}^2 K_{ij} e^{u_j}
  - 4\pi \sum_{j=1}^N \delta_{p_j}
  \quad \text{on } M, \quad i=1,2,
\ee
where $K=(K_{ij})$ is the coupling matrix, $\lambda>0$ is the coupling parameter, $N$ is the total multiplicity of the vortex points $p_1,\dots,p_N$, $M$ is a compact surface without boundary endowed with a Riemannian metric $g$, and $\Delta$ denotes the Laplace–Beltrami operator associated with $g$.

In the present work we are interested in the case where the coupling matrix $K$ has vanishing diagonal entries and nontrivial off-diagonal entries, namely 
\be \label{k}
K=\left(
\begin{array}{cc}
0 & 1 \\
1 & 0
\end{array}
\right).
\ee
From a physical point of view, this corresponds to what Dziarmaga \cite{dzia1} calls “purely mutual statistical interaction”: the effective interaction matrix between the two species has zero self-coupling, so that each component does not interact with itself but interacts only with the other one. In particular, the matrix \eqref{k} arises in the non-relativistic $[U(1)]^2$ Chern–Simons model coupled to two complex
scalar fields considered in \cite{dzia1}.

\medskip

For the reader’s convenience we briefly recall in Section~\ref{sec:self-dual}
how the self-dual reduction of the non-relativistic $[U(1)]^2$ Chern–Simons
model with two complex scalar fields, as considered in \cite{dzia1,klkl},
leads to the elliptic system \eqref{system0}.

\

We will actually consider the following more general problem in mean field formulation  
\be \label{system}
\left\{
\begin{array}{l}
-\Delta u_1= \ds{\rho_2\left(\frac{h_2\re^{u_2}}{\int_M h_2\re^{u_2}\,dV_g}-1\right)-4\pi\sum_{j=1}^{N}\a_j\left(\delta_{p_j}-1\right),}\vspace{0.2cm}\\
-\Delta u_2=\ds{\rho_1\left(\frac{h_1\re^{u_1}}{\int_M h_1\re^{u_1}\,dV_g}-1\right)-4\pi\sum_{j=1}^{N}\a_j\left(\delta_{p_j}-1\right),}
\end{array}
\right.
\ee
where $\rho_1, \rho_2$ are real positive parameters, $h_1, h_2$ smooth positive functions, $\a_j\geq0$ and $dV_g$ is the volume form. For the sake of
simplicity, we are normalizing the total volume of $M$ so that $|M|=1$. 

Observe that when system \eqref{system} is reduced to a single equation we recover the following well-known singular mean field equation 
\be \label{mf}
-\Delta u= \rho\left(\frac{h\re^u}{\int_M h\re^u\,dV_g}-1\right)-4\pi\sum_{j=1}^{N}\a_j\left(\delta_{p_j}(x)-1\right),
\ee
which has been widely studied in literature, see for example \cite{BdMM,bgjm,bjly,bjly2,bjly4,bjl,bl,bt,bt2,bm,clmp2,cama,cl1,cl2,CLin4,cl4,EGP,GM1,Troy}.

Let us now return to system \eqref{system} and discuss its analytical aspects. Since it is invariant by the addition of constants we will restrict ourselves to the subspace of functions with zero average
$$
	\ov H^1(M)= \left\{ u\in H^1(M) \, : \, \int_M u\,dV_g=0 \right\}.
$$
We next desingularize the system \eqref{system} by considering the Green function $G_p(x)$ on $M$ with pole at $p\in M$,
$$
	-\D G_p(x)=\d_p -1, \quad \int_M G_p\,dV_g=0,
$$
and by performing the substitution 
\begin{align}
\begin{split}\label{change}
	&u_i(x)\mapsto u_i(x)+4\pi\sum_{j=1}^{N} G_{p_j}(x), \\
	& h_i(x) \mapsto \wtilde h_i(x)=h_i(x)\re^{-4\pi\sum_{j=1}^{N} \a_jG_{p_j}(x)}, 
\end{split}	
\end{align}
so that \eqref{system} is rewritten as 
\be \label{system2}
\left\{
\begin{array}{l}
-\Delta u_1= \ds{\rho_2\left(\frac{\wtilde h_2\re^{u_2}}{\int_M \wtilde h_2\re^{u_2}\,dV_g}-1\right),}\vspace{0.2cm}\\
-\Delta u_2=\ds{\rho_1\left(\frac{\wtilde h_1\re^{u_1}}{\int_M \wtilde h_1\re^{u_1}\,dV_g}-1\right),}
\end{array}
\right.
\ee
where $\wtilde h_i$ are such that
\begin{align} \label{h}
	& \wtilde h_i> 0 \quad \mbox{on } M\setminus \{p_1,\dots, p_{N}\}, \qquad \wtilde h_i(x) \simeq d(x,p_j)^{2\a_j} \quad \mbox{near } p_j.
\end{align}
The latter problem has a variational structure and the associated Euler-Lagrange functional $\mathcal{J}_\rho :\ov H^1(M)\times \ov H^1(M)\to \R$, $\rho=(\rho_1,\rho_2)$, is given by
\be \label{funct}
	\mathcal{J}_\rho(u_1,u_2)= \int_M \n u_1\cdot \n u_2\,dV_g -\rho_2 \log\int_M \wtilde h_2\re^{u_2}\,dV_g -\rho_1 \log \int_M \wtilde h_1\re^{u_1}\,dV_g.
\ee
Observe that the latter functional is indefinite which makes the problem hard to attack variationally, see for example \cite{br, lpy, yana1, yana2, yana3} and the reference therein. We will overcome this issue by exploiting a suitable constrained problem which will permit us to carry out a Morse-theoretical argument yielding existence and multiplicity results for any underlying surface $M$, see the discussion below.

\

Before stating the main result let us recall what is known about \eqref{system}. The only existence result seems to be derived in \cite{gl, lz}, where the following general problem is considered
\be \label{gen}
	-\D u_i = \sum_{j=1}^n \rho_j a_{ij}\left( \frac{\wtilde h_j \re^{u_j}}{\int_M \wtilde h_j \re^{u_j}\,dV_g} - 1 \right), 	\quad i=1,\dots,n. 
\ee
Here $\rho_i$ are real positive parameters, $a_{ij}$ real numbers and $\wtilde h_i$ are as in \eqref{h}. The associated Euler-Lagrange functional is given by
$$
	\mathcal{G}_\rho(u)=\frac 12 \sum_{i,j=1}^n a^{ij} \int_M \n u_i \cdot \n u_j \,dV_g - \sum_{j=1}^n \rho_j \log \int_M \wtilde h_j \re^{u_j}\,dV_g,
$$
where $(a^{ij})$ is the inverse matrix of $A=(a_{ij})$. In \cite{gl,lz} the coupling matrix $A$ is assumed to satisfy the following hypotheses:

\medskip

\noindent (H1): $A$ is symmetric, nonnegative, irreducible and invertible,

\medskip

\noindent (H2):  $a^{ii}\leq 0$ for all $i=1,\dots,n$, $a^{ij}\geq 0$ for all $i\neq j$ and $\sum_{j=1}^n a^{ij}\geq 0$ for all $i=1,\dots,n$.

\medskip

Here $A$ is called nonnegative if $a_{ij}\geq 0$ for all $i,j=1,\dots,n$ and is called irreducible if there is no $J\subset\{1,\dots,n\}$ such that $a_{ij}=0$ for all $i\in J$ and $j\notin J$. For $n=2$, 
$$
A=\left( \begin{array}{cc}
							a_{11} & a_{12} \\
							a_{12} & a_{22}
				 \end{array}\right)
$$
satisfies (H1) and (H2) if and only if $a_{ij}\geq 0$, $\max\{ a_{11},a_{22} \}\leq a_{12}^2$ and det$A\neq0$. Observe that this is the case for the coupling matrix $K$ in \eqref{k} relative to the system \eqref{system}.

\medskip

The blow-up phenomenon may be quite complicate in general for systems alike \eqref{gen}, see for example \cite{lz2}. The main contribution of \cite{gl,lz} is to show that under the assumptions (H1), (H2) such complexity can be avoided. Indeed, they proved that any blow-up point carries the same local blow-up mass and thus the critical set of parameters associated to \eqref{gen} can be determined as follows. Set
\begin{align*}
\Sigma &= \bigg\{ 8\pi m + \sum_{j\in J}8\pi(1+\a_j) \,:\, m\in\N\cup\{0\}, \, J\subseteq\{1,\dots,N\}  \bigg\}\setminus\{0\} \\
			&=\bigg\{8\pi n_k \,:\, n_1<n_2<\dots\bigg\},
\end{align*}
which corresponds to the critical set of the standard Liouville equation \eqref{mf}. Observe that if either $\a_j=0$ or $\a_j\in\N$ for all $j$ then we simply have $\Sigma=8\pi\N$. It turns out that the critical set for the system \eqref{gen} can be expressed as
$$
	\G = \left\{ (\rho_1,\dots,\rho_n)\in\R^n_+ \,:\, \sum_{i,j} a_{ij}\rho_i\rho_j=8\pi n_k \sum_{i} \rho_i \quad \mbox{ for some } k\in\N \right\}.
$$
For example, the critical set associated to our problem \eqref{system} is given by
\be \label{L}
	\L = \left\{ (\rho_1,\rho_2)\in\R^2_+ \,:\, \dfrac{\rho_1\rho_2}{\rho_1+\rho_2}=4\pi n_k \quad \mbox{ for some } k\in\N \right\},
\ee 
see Figure \ref{fig:L} for the case $\a_j=0$ or $\a_j\in\N$ for all $j$.

\begin{figure}[h]
\centering
\includegraphics[width=0.6\linewidth]{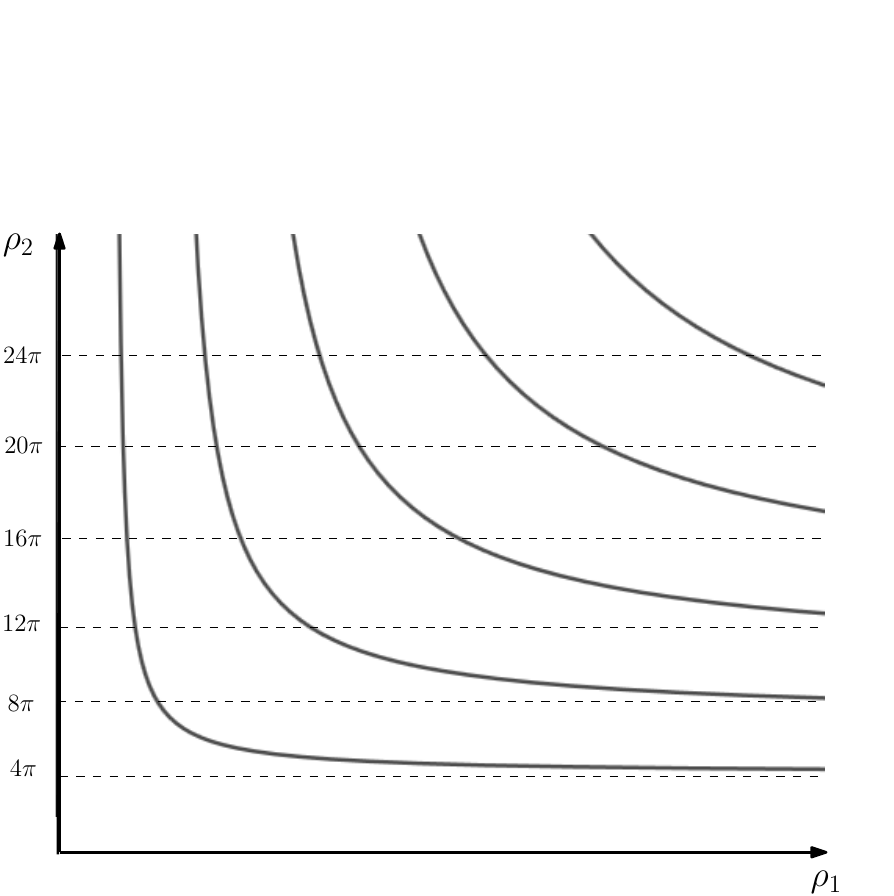}
\caption{}
\label{fig:L}
\end{figure}

Once the critical set is known, the authors in \cite{gl,lz} derive the Leray-Schauder degree of \eqref{gen} by deforming it into a single equation alike \eqref{mf} for which the degree is well-known. We collect these results in the following theorem.

\

\noindent \textbf{Theorem A} (\cite{gl,lz})\textbf{.} \emph{Let $A=(a_{ij})$ satisfy (H1) and (H2). Then, we have:}

\medskip

\noindent \textbf{1.} \emph{If $(\rho_1,\dots,\rho_n)\in K$, where $K\subset \R^2\setminus\G$ is a compact set, then there exists a constant $C>0$ such that any solution $(u_1,\dots,u_n)$ to \eqref{gen} satisfies
$$
|u_i(x)|\leq C \quad i=1,\dots,n, \quad x\in M.
$$} 

\noindent \textbf{2.} \emph{Let $d_k$ be the Leray-Schauder degree associated to \eqref{gen} for 
$$
8\pi n_k\sum_{i}\rho_i<\sum_{i,j} a_{ij}\rho_i\rho_j <8\pi n_{k+1}\sum_{i}\rho_i.
$$
Then, $d_k$ coincides with the degree associated to \eqref{mf} for $8\pi n_k<\rho<8\pi n_{k+1}$. In particular, if $(\rho_1,\dots,\rho_n)\notin\G$, all $\a_j\in\N$ and $M$ has positive genus, then $d_k>0$ and \eqref{gen} has a solution.}

\

On the other hand, we are not aware of any result concerning either $M=\S^2$ or $\a_j\notin\N$. The aim of this paper is to give the following existence result in this direction, complementing it with a multiplicity result. 

\smallskip

\begin{thm} \label{thm1}
Suppose $(\rho_1,\rho_2)\notin\L$ and that
\begin{align*}
 \rho_1,\rho_2>8k\pi, \quad \frac{\rho_1+\rho_2}{2} < 8(k+1)\pi.
\end{align*}
Then, we have:

\medskip

\noindent \textbf{1.} If $M=\S^2$ and all $\a_j=0$, then \eqref{system} has a solution.

\medskip

\noindent \textbf{2.} If $M$ has positive genus $\emph{\textbf{g}}>0$ and $\a_j\geq0$, then for a generic choice of the metric $g$ and of the functions $h_1,h_2$ it holds
$$
 \#\bigr\{ \mbox{solutions of \eqref{system}} \bigr\} \geq \binom{k+\emph{\textbf{g}}-1}{\emph{\textbf{g}}-1}. 
$$
\end{thm}
Here, by generic choice of $(g,h_1,h_2)$ we mean that it can be taken in an open dense subset of $\M^2 \times C^{\,2}(M)\times C^{\,2}(M)$, where $\M^2$ stands for the space of Riemannian metrics on $M$ equipped with the $C^{\,2}$ norm.  

\medskip


The proof is based on the variational structure of the problem. As already observed, we handle the indefiniteness of the energy functional \eqref{funct} by introducing a suitable constrained problem, see \eqref{funct2}. In particular, the purely mutual interaction encoded by the off-diagonal matrix $K$ gives \eqref{system0} a variational structure (skew-symmetric after a transformation) that is closely related to that of the Chern–Simons–Higgs systems studied in \cite{lpy,lp}. In our setting, the constrained functional $\widetilde{J}_\rho$ plays the role of the reduced functional used in those works. We will then run a Morse-theoretical argument on the constrained problem by looking at topological changes in the structure of sublevels of $\wtilde{J}_\rho$. More precisely, we first show that high sublevels have trivial topology. Then, for the case $M=\S^2$ and $\alpha_j=0$, we will prove that low sublevels are homotopically equivalent to some formal barycenters of $\S^2$ (similarly to \cite{dj}), which are not contractible, thus yielding the first part of Theorem~\ref{thm1}. On the other hand, for $M$ with positive genus and $\alpha_j\geq0$, we need to avoid the non-trivial effect of the singular points $p_j$. This is done in the spirit of \cite{BdMM,BJMR}, by first retracting $M$ onto a curve not intersecting those points and then showing that formal barycenters of such a curve embed in low sublevels of $\wtilde{J}_\rho$, yielding the second part of Theorem~\ref{thm1}.

\medskip

The paper is organized as follows. In section \ref{sec:self-dual} we review the Chern-Simons model proposed in \cite{dzia1, klkl}, obtain the self-dual equations and derive the singular Liouville system \eqref{system}, in section \ref{sec:existence} we prove the main existence result, Theorem \ref{thm1}.

\

\section{Self-dual equations} \label{sec:self-dual}




In this section we review the non-relativistic Chern-Simons model proposed in \cite{dzia1, klkl}, obtain the self-dual equations and derive the singular Liouville system \eqref{system}. The Lagrangian density of this model is given by
$$
\mathcal{L}=\kappa\epsilon^{\mu\nu\alpha}A_\mu\partial_\nu \hat{A}_\alpha+\ri \overline{\phi}D_0\phi+\ri \overline{\psi}D_0\psi
-|D_k\phi|^2-|D_k\psi|^2+g|\phi|^2|\psi|^2, 
$$
where the covariant derivatives are defined by
$$
 D_\mu\phi=\partial_\mu\phi-\ri A_\mu\phi,\quad D_\mu\psi=\partial_\mu\psi-\ri \hat{A}_\mu\psi
$$
and $\mu,\nu,\alpha \in \{0,1,2\}$, $k\in\{1,2\}$, while
$\epsilon^{\mu\nu\alpha}$ is the totally antisymmetric tensor with
$\epsilon^{012}=1$. 
The constant $\kappa>0$ is the Chern–Simons coupling constant. 
 The time-derivative terms here and in \cite{dzia1} differ only by a total time derivative, so the Euler–Lagrange equations in the interior coincide, and in particular the resulting self-dual static system \eqref{system} is unchanged.

Varying $\mathcal{L}$ with respect to $A_0$ and $\hat{A}_0$ yields the Gauss laws
\begin{align} \label{gauss}
 \kappa F_{12}=-|\psi|^2, \qquad
 \kappa \hat{F}_{12}=-|\phi|^2,
\end{align}
where
\[
  F_{\mu\nu} = \partial_\mu A_\nu - \partial_\nu A_\mu,
  \qquad
  \hat{F}_{\mu\nu} = \partial_\mu \hat{A}_\nu - \partial_\nu \hat{A}_\mu,
  \qquad \mu,\nu = 0,1,2.
\]
Varying the action with respect to $\phi$, $\psi$ and the spatial components of $A_\mu$, $\hat{A}_\mu$, we obtain the following equations of motion: 
\begin{equation*}\begin{aligned}
\ri D_0\phi+ D_j^2\phi+g|\psi|^2\phi=0, &\qquad \ri D_0\psi+ D_j^2\psi+g|\phi|^2\psi=0,\\
 \kappa F_{j0}=-\epsilon^{0jk}\hat{J}^k, & \qquad \kappa \hat{F}_{j0}=-\epsilon^{0jk}J^k, \qquad j=1,2,
\end{aligned}\end{equation*}
where the current densities  $J^k, \hat{J}^k$ are defined by
\[
  J^k=\ri(\phi \overline{D^k\phi}-\overline{\phi} D^k\phi), \quad \hat{J}^k=\ri(\psi \overline{D^k\psi}-\overline{\psi} D^k\psi).
\]
In view of the Gauss laws  \eqref{gauss} and the following identities
\ber
 |D_k\phi|^2&=&|D_1\phi\pm\ri D_2\phi|^2\mp\ri\left(\partial_1\big[\phi\overline{D_2\phi}\big]-\partial_2\big[\phi\overline{D_1\phi}\big]\right)\pm F_{12}|\phi|^2,\nm\\
  |D_k\psi|^2&=&|D_1\psi\pm\ri D_2\psi|^2\mp\ri\left(\partial_1\big[\psi\overline{D_2\psi}\big]-\partial_2\big[\psi\overline{D_1\psi}\big]\right)\pm \hat{F}_{12}|\psi|^2,\nm
 \eer
 the static Hamiltonian density can be written as
\ber
\mathcal{H}&=&-\mathcal{L} \quad \text{(up to a total divergence)} \nm\\
 &=&-\kappa A_0\hat{F}_{12}-\kappa \hat{A}_0F_{12}-A_0|\phi|^2-\hat{A}_0|\psi|^2+|D_k\phi|^2+|D_k\psi|^2- g|\phi|^2|\psi|^2\nm\\
 &=&|D_k\phi|^2+|D_k\psi|^2-g|\phi|^2|\psi|^2\nm\\
 &=&|D_1\phi\pm\ri D_2\phi|^2
 +|D_1\psi\pm\ri D_2\psi|^2 \mp \ri\left(\partial_1\big[\phi\overline{D_2\phi}\big]-\partial_2\big[\phi\overline{D_1\phi}\big]\right)\nm\\
 && \mp \ri\left(\partial_1\big[\psi\overline{D_2\psi}\big]-\partial_2\big[\psi\overline{D_1\psi}\big]\right)
 \pm F_{12}|\phi|^2\pm \hat{F}_{12}|\psi|^2 - g|\phi|^2|\psi|^2\nm\\
  &=&|D_1\phi\pm\ri D_2\phi|^2
 +|D_1\psi\pm\ri D_2\psi|^2 \mp \ri\left(\partial_1\big[\phi\overline{D_2\phi}\big]-\partial_2\big[\phi\overline{D_1\phi}\big]\right)\nm\\
 &&\mp \ri\left(\partial_1\big[\psi\overline{D_2\psi}\big]-\partial_2\big[\psi\overline{D_1\psi}\big]\right)
 -\left(g\pm\frac2\kappa\right)|\phi|^2|\psi|^2. \nm
 \eer

At the critical coupling $g=\mp\frac2\kappa$ the energy  satisfies
$$
E=\int\mathcal{ H}\ud x\ge 0
$$
and the minimum is attained only if there hold the following self-dual equations
\be \label{self-dual}
\left\{
\ba{l}
 D_1\phi\pm\ri D_2\phi=0,\\
 D_1\psi\pm\ri D_2\psi=0,
\ea
\right.
\ee
together with the Gauss laws \eqref{gauss}. It follows from \eqref{self-dual} that the zeros of $\phi$ and $\psi$ are discrete and finite, which we suppose to coincide and we denote $\{p_1, \dots, p_{N}\}$. 
Let $u_1=\ln|\phi|^2$ and $u_2=\ln|\psi|^2$. We transform the system \eqref{gauss} and \eqref{self-dual} into
$$
\left\{
\ba{l}
-\Delta u_1=\ds{\mp\frac2\kappa\re^{u_2}-4\pi\sum_{j=1}^{N}\delta_{p_j}(x),} \vspace{0.2cm}\\
-\Delta u_2=\ds{\mp\frac2\kappa\re^{u_1}-4\pi\sum_{j=1}^{N}\delta_{p_j}(x),} 
\ea
\right.
$$
which is solvable only for the choice $\mp\frac2\kappa=\frac2\kappa$. Therefore, we consider 
$$
\left\{
\ba{l}
-\Delta u_1=\ds{\lambda\re^{u_2}-4\pi\sum_{j=1}^{N}\delta_{p_j}(x),} \vspace{0.2cm}\\
-\Delta u_2=\ds{\lambda\re^{u_1}-4\pi\sum_{j=1}^{N}\delta_{p_j}(x),} 
\ea
\right.
$$
where $\lambda\equiv \frac{2}{\kappa}>0.$ Actually, we will study a more general version of the above system on a compact surface, namely \eqref{system}. For more details concerning the relation between Liouville type equations and self-dual equations, we refer the interested readers to \cite{dunne, tarantello}.

\

\section{Proof of the main results} \label{sec:existence}

\medskip

In this section we derive the main existence and multiplicity results of Theorem \ref{thm1} by making use of a variational approach. We will first introduce a suitable constrained problem and then run the Morse-theoretical argument, treating separately the zero and positive genus cases.

\medskip

\subsection{A constrained problem} 

In this subsection, we introduce the constrained problem and discuss its properties. To this end, we first consider a change of variables in order to decouple the mixed term in the Euler-Lagrange functional $\mathcal{J}_\rho$. Let $F, G\in\ov H^1(M)$ and set
\begin{align*}
\left\{
\begin{array}{l}
u_1=F-G,\\
u_2=F+G.
\end{array}
\right.	
\end{align*}
Then, in the new variables, \eqref{funct} takes the form
\begin{align*}
\mathcal{J}_\rho(u_1,u_2)=J_\rho(F,G)=& \int_M |\n F|^2 \,dV_g - \int_M |\n G|^2 \,dV_g  \\
																 &-\rho_2 \log\int_M \wtilde h_2\re^{F+G}\,dV_g -\rho_1 \log \int_M  \wtilde h_1\re^{F-G}\,dV_g.
\end{align*}
The idea is now to rewrite it as
$$
J_\rho(F,G)= \int_M |\n F|^2 \,dV_g -I_\rho(F,G),
$$
where
\be \label{I}
I_\rho(F,G)= \int_M |\n G|^2 \,dV_g+\rho_2 \log\int_M \wtilde h_2\re^{F+G}\,dV_g +\rho_1 \log \int_M \wtilde h_1\re^{F-G}\,dV_g
\ee
and consider the constrained functional
\be \label{funct2}
	\wtilde J_\rho(F)=\int_M |\n F|^2 \,dV_g -I_\rho\bigr(F,\wtilde G(F)\bigr),
\ee
where $\wtilde G(F)$ is defined by
\be \label{min}
 I_\rho(F,\wtilde G(F)) = \min \left\{ I_\rho(F,G) \, : \, G\in \ov H^1(M) \right\}. 
\ee
We point out that a similar argument has been already used in \cite{lpy, lp} where a relativistic Chern-Simons model was studied. The definition of \eqref{funct2} is well-posed, indeed we have the following easy lemma.

\smallskip

\begin{lem}
For any $F\in\ov H^1(M)$ fixed there is a unique $G=\wtilde G(F)\in\ov H^1(M)$ satisfying \eqref{min}.
Moreover, the minimizer $\wtilde G(F)$ is $C^1$ with respect to $F$.
\end{lem} 

\begin{proof}
Indeed, by the Moser-Trudinger inequality $I_\rho(F,\cdot)$ is weakly lower semicontinuous in $\ov H^1(M)$. Moreover, since $F, G\in \ov H^1(M)$ we have by the Jensen and Poincar\'e inequalities that
$$
	I_\rho(F,G) \geq \int_M |\n G|^2 \,dV_g -C,
$$ 
where $C>0$ depends only on $\wtilde h_1, \wtilde h_2$, $\rho_1, \rho_2$ and $M$. Since $\int_M |\n \cdot|^2 \,dV_g$ induces a norm on $\ov H^1(M)$, we conclude that $I_\rho(F,\cdot)$ is coercive. It follows that it has a global minimum which is unique once we observe that $I_\rho(F,\cdot)$ is convex. 
Indeed, it holds that
\begin{align*}
D_{GG}I_\rho(F,G)[\phi,\phi]=&\int_M |\nabla \phi|^2 dV_g  \\
&+ \rho_2 \frac{\int_M \tilde h_2 e^{F+G}\phi^2 dV_g \int_M \tilde h_2 e^{F+G} dV_g - (\int_M \tilde h_2 e^{F+G}\phi dV_g)^2}{(\int_M \tilde h_2 e^{F+G}dV_g)^2}\\
&+\rho_1 \frac{\int_M \tilde h_1 e^{F-G}\phi^2 dV_g \int_M \tilde h_1 e^{F-G} dV_g - (\int_M \tilde h_1 e^{F-G}\phi dV_g)^2}{(\int_M \tilde h_1 e^{F-G}dV_g)^2} \\
\ge & \int_M |\nabla \phi|^2 dV_g,
\end{align*}
and $I_\rho(F,\cdot)$ is convex.
Moreover, by Lax-Milgram theorem, we have $D_{GG}I_\rho(F,\wtilde G):\ov H^1(M) \to (\ov H^1(M))^*$ is invertible and its inverse is bounded.
Thus, by the implicit function theorem, $\widetilde G(F)$ is continuously differentiable with respect to $F$.
\end{proof}

\medskip

It is easy to check that a critical point of $\wtilde J_\rho$ gives rise to a solution of the original problem \eqref{system}. Thus, from now on, we focus on the latter functional trying to detect a change of topology between its sublevels. We will use the notation
$$
\wtilde J_\rho^a =\{F\in\ov H^1(M) \,:\, \wtilde J_\rho(F)\leq a\}.
$$ 
The starting point is the following compactness result, which is already part of Theorem A.

\smallskip

\begin{pro} \emph{(\cite{bl,lz}).} \label{pro:comp}
Let $\L$ be as in \eqref{L} and suppose $(\rho_1,\rho_2)\notin\L$. Then, the solutions of \eqref{system} are uniformly bounded in $C^2(M)$.
\end{pro}

\smallskip

The latter property is used to bypass the Palais-Smale condition, which is not known for $\wtilde J_\rho$. Indeed, one can use it jointly with the arguments of \cite{lucia}, where a deformation lemma for a class of functionals alike \eqref{funct2} is constructed, to derive the following.

\smallskip

\begin{lem} \label{lem:def}
Suppose $(\rho_1,\rho_2)\notin\L$ and that $\wtilde J_\rho$ has no critical levels inside some interval $[a,b]$. Then, $\wtilde J_\rho^a$ is a deformation retract of $\wtilde J_\rho^b$.
\end{lem}

\smallskip

Indeed, for the validity of the above result one needs to check some compactness properties of the functional $I_\rho\bigr(F,\wtilde G(F)\bigr)$ in \eqref{funct2}. In our case, we just observe that whenever we have a bounded sequence $\|F_n\|_{\ov H^1(M)}\leq C$, then by construction of \eqref{min}, we have $I_\rho\bigr(F_n,\wtilde G(F_n)\bigr)\leq I_\rho(F_n,0)$. 
Hence, we obtain
\begin{align*}
\int_M |\n \wtilde G(F_n)|^2 \,dV_g &\leq I_\rho\bigr(F_n,\wtilde G(F_n)\bigr) + C \leq I_\rho(F_n,0) + C \\
																		&\leq C_1 \int_M |\n F_n|^2 \,dV_g+ C_2,
\end{align*}
Here, in the last inequality, we have used the Moser-Trudinger inequality. We also conclude that $\|\wtilde G(F_n)\|_{\ov H^1(M)}\leq C$. 
Thus, up to a subsequence, $F_n\rightharpoonup F$ and $\wtilde G(F_n) \rightharpoonup  G$ weakly in $\ov H^1(M)$.
Then, by the Moser-Trudinger inequality, we have $\re^{F_n}\to\re^{F}$ and $\re^{\wtilde G(F_n)}\to\re^{G}$ strongly in any $L^p(M)$. 
With this at hand, we can adapt the arguments of \cite{lucia} to our setting to obtain Lemma \ref{lem:def}.

\medskip

Observe next that, by Proposition \ref{pro:comp}, $\wtilde J_\rho$ has no critical points above a sufficiently large level $b\gg0$. Therefore, the deformation in Lemma \ref{lem:def} can be used to obtain a deformation retract of $\ov H^1(M)$ onto the sublevel $\wtilde J_\rho^b$ as in Corollary 2.8 in \cite{mal}.

\smallskip

\begin{pro} \label{pro:contr}
Suppose $(\rho_1, \rho_2)\notin\L$. Then, there exists $L > 0$ sufficiently
large such that $\wtilde J_\rho^L$ is a deformation retract of $\ov H^1(M)$. In particular, it
is contractible.
\end{pro}

\smallskip

The existence results of the main Theorem \ref{thm1} will then follow by showing the low sublevels of $\wtilde J_\rho$ are not contractible. Indeed, we will see that low sublevels are related to some family of unit measures supported in (at most) $k$ points of $M$, known as the formal barycenters of $M$ of order $k$:
\begin{equation}\label{sigk}
	M_k = \left\{ \sum_{i=1}^k t_i\delta_{x_i} \, : \, \sum_{i=1}^k t_i=1,t_i\geq 0,x_i\in M,\forall\,i=1,\dots,k \right\},
\end{equation}
as it happens for the standard Liouville equation \eqref{mf}; see, for example, \cite{BdMM,dj}. We will start by showing that there exists a map
$$
	\Psi: \wtilde J_\rho^{-L}\to M_k,
$$
for some $L > 0$ sufficiently large and $\frac{\rho_1+\rho_2}{2} < 8(k+1)\pi$. This is done by means of improved versions of the Moser-Trudinger inequality
\be \label{mt}
	2\log \int_M \widehat h\,\re^F\,dV_g \leq \frac{1}{8\pi} \int_M |\n F|^2\,dV_g+C, \quad F\in \ov H^1(M),
\ee
where $C$ is a constant depending only on $M$ and
$$
 \widehat h(x)=\re^{-4\pi\sum_{j=1}^{N} \a_jG_{p_j}(x)} 
$$
is the singular part of $\wtilde h_i$, so that $\wtilde h_i(x)=h_i(x)\widehat h(x)$. (See \eqref{change}). The inequality \eqref{mt} has been first derived for the regular case ($\a_j=0$ for all $j$), but it remains valid with the singular weight $\widehat h$ since $\widehat h$ is uniformly bounded by $\a_j\geq0$, $j=1,\cdots, N$. 
For such an inequality, one has the following improvement, provided that the conformal volume $\widehat h\,\re^F\,dV_g$ is suitably spread out. For the proof, we refer to \cite{chenli,dj}.

\smallskip

\begin{lem} \label{lem:mt}
Let $l$ be an integer and let $\O_1,\dots,\O_l$ be subsets of $M$ satisfying $d(\O_i,\O_j)\geq\d$ for any $i\neq j$, for some $\d>0$. Then, for any $\g>0$ and for any $\e>0$, there exists $C=C(M,l,\d,\g,\e)$ such that
$$
	2\log \int_M \widehat h\,\re^F\,dV_g \leq \frac{1+\e}{8l\pi} \int_M |\n F|^2\,dV_g+C,
$$
for any $F\in \ov H^1(M)$ such that
\be \label{spread}
	\int_{\O_i} \frac{\widehat h\,\re^{F}\,dV_g }{\int_M \widehat h\,\re^{F}\,dV_g} \geq \g, \quad j=1,\dots,l.
\ee
\end{lem}

\medskip

We now explain how to use this result in our setting. To this end, we start with the following estimate. 

\smallskip

\begin{lem} \label{lem:est1}
For any $F\in\ov H^1(M)$ we have
$$
	\wtilde J_\rho(F)\geq\int_M |\n F|^2 \,dV_g-(\rho_1+\rho_2)\log \int_M \widehat h\,\re^F \,dV_g -C,
$$
where $C>0$ depends only on $h_1, h_2$, $\rho_1, \rho_2$ and $M$.
\end{lem}

\begin{proof}
Observe that, by definition \eqref{min}, we have $I_\rho\bigr(F,\wtilde G(F)\bigr)\leq I_\rho(F,0)$ for any $F\in \ov H^1(M)$. Hence,
\begin{align*}
	\wtilde J_\rho(F) &=\int_M |\n F|^2 \,dV_g -I_\rho\bigr(F,\wtilde G(F)\bigr) \geq \int_M |\n F|^2 \,dV_g -I_\rho(F,0) \\
										&\geq\int_M |\n F|^2 \,dV_g-(\rho_1+\rho_2)\log \int_M \widehat h\,\re^F \,dV_g -C,
\end{align*}
where $C>0$ depends only on $h_1, h_2$, $\rho_1, \rho_2$, $M$ and we are using $\frac1C \leq h_i \leq C$.
\end{proof}

\medskip

We can now state the crucial property.

\smallskip

\begin{lem}
Suppose $\frac{\rho_1+\rho_2}{2} < 8(k+1)\pi$. Then, for any $\e,r>0$, there exists $L=L(\e,r)>0$ such that for any $F\in\wtilde J_\rho^{-L}$ there exist $k$ points $\{p_1,\dots,p_k\}\subset M$ such that
$$
	\int_{\bigcup_{i=1}^k B_r(p_i)} \frac{\widehat h\,\re^{F}\,dV_g }{\int_M \widehat h\,\re^{F}\,dV_g} \geq 1-\e.
$$
\end{lem}
\begin{proof}
Suppose by contradiction that the statement is not true. It means that there exist $\e,r>0$ and $F_n\in\ov H^1(M)$ with 
\be \label{inf}
\wtilde J_\rho(F_n)\to-\infty
\ee 
such that for any $k$-tuple $\{p_1,\dots,p_k\}\subset M$ there holds $\int_{\bigcup_{i=1}^k B_r(p_i)} \frac{\widehat h\,\re^{F}\,dV_g }{\int_M \widehat h\,\re^{F}\,dV_g} < 1-\e$. Then, by a covering argument as in Lemma 3.3 in \cite{dj}, there exist $k+1$ subsets $\O_1,\dots,\O_{k+1}$ of $M$ satisfying the volume spreading condition \eqref{spread} with $F=F_n$ and $l=k+1$. By Lemma \ref{lem:est1} and the improved inequality in Lemma \ref{lem:mt}, it follows that
\begin{align*}
		\wtilde J_\rho(F_n) & \geq\int_M |\n F_n|^2 \,dV_g-(\rho_1+\rho_2)\log \int_M \widehat h\,\re^{F_n} \,dV_g -C \\
											& \geq \left(1-\frac{\rho_1+\rho_2}{2}\frac{1+\e}{8(k+1)\pi}\right)\int_M |\n F_n|^2 \,dV_g-C \\
											& \geq -C
\end{align*}
for $\e>0$ sufficiently small, where in the last inequality we have used $\frac{\rho_1+\rho_2}{2} < 8(k+1)\pi$. This is clearly in contradiction with \eqref{inf}.
\end{proof}

\medskip

The above lemma shows that the measure $\frac{\widehat h\,\re^{F}\,dV_g }{\int_M \widehat h\,\re^{F}\,dV_g}$ resembles a measure supported in (at most) $k$ points. Then, it is not difficult to project such measure on the closest element of the formal barycenters $M_k$. We refer to Lemma 4.9 in \cite{dj} for the proof.

\smallskip

\begin{pro} \label{pro:proj}
Suppose $\frac{\rho_1+\rho_2}{2} < 8(k+1)\pi$. Then, there exists a projection $\Psi: \wtilde J_\rho^{-L}\to M_k$.
\end{pro}

\

To conclude the characterization of the low sublevels, we need to construct a reverse map: more precisely, we will show that there exists a map
$$
	\Phi: X\subseteq M_k\to \wtilde J_\rho^{-L},
$$
for a suitable subset $X$ to be chosen later and $\rho_1,\rho_2>8k\pi$. To this end, let us introduce the following family of test functions modeled on a sum of regular bubbles. Given $\s\in M_k$, $\s=\sum_{i=1}^k t_i\d_{x_i}$ and $\l>0$ we set $\varphi_{\l,\s}:M\to\R$
\be \label{test}
	\varphi_{\l,\s}(y) = \log \sum_{i=1}^k t_i \left( \frac{\l}{1+\l^2 d(y,x_i)^2} \right)^2 -\log\pi.
\ee
For this class of functions, the following property holds true as in Proposition 4.2 in \cite{mal2}.

\smallskip

\begin{lem} \label{lem:test}
Let $K$ be a compact subset of $M\setminus\{p_1,\dots,p_N\}$, $K_k\subset M_k$ be the formal barycenters supported in $K$ and let $\varphi_{\l,\s}$, $\s\in K_k$ be the functions defined above. Then, for $\l\to+\infty$ we have
\begin{align*}
	&\frac12 \int_M |\n \varphi_{\l,\s}|^2 \,dV_g = 16k\pi(1+o_\l(1))\log\l, \\
	&\log \int_M \widehat h\,\re^{\varphi_{\l,\s}-\ov{\varphi}_{\l,\s}} \,dV_g = 2(1+o_\l(1))\log\l,
\end{align*}
where $\ov{\varphi}_{\l,\s}$ is the average of ${\varphi}_{\l,\s}$ and $o_\l(1)\to0$ as $\l\to+\infty$. Moreover,
\be \label{conv}
\frac{\re^{\varphi_{\l,\s}}}{\int_M \re^{\varphi_{\l,\s}}} \rightharpoonup \s\in K_k,
\ee
weakly in the sense of measures.
\end{lem}

\smallskip

The latter estimates have been obtained for the regular case ($\a_j=0$ for all $j$), but they continue to hold true  with the singular weight $\widehat h$ as far as the bubbles are centered away from the singular points $\{p_1,\dots,p_N\}$. 

We will then need the following result. 

\smallskip

\begin{lem} \label{lem:est2}
Suppose $\rho_1\leq\rho_2$. Then, for any $F\in\ov H^1(M)$ we have
$$
	\wtilde J_\rho(F) \leq \int_M |\n F|^2 \,dV_g -2\rho_1 \log \int_M \widehat h\,\re^F\,dV_g+C,
$$
where $C>0$ depends only on $h_1, h_2$, $\rho_1, \rho_2$ and $M$.
\end{lem}

\begin{proof}
Since $F,\wtilde G \in \ov H^1(M)$, by Jensen's inequality, it holds that
\begin{align} \label{est}
\begin{split}
		\wtilde J_\rho(F) & = \int_M |\n F|^2 \,dV_g- \int_M |\n \wtilde G|^2 \,dV_g 	\\
		                  & \quad -\rho_2 \log\int_M \wtilde h_2\re^{F+\wtilde G}\,dV_g -\rho_1 \log \int_M  \wtilde h_1\re^{F-\wtilde G}\,dV_g \\
						& \leq \int_M |\n F|^2 \,dV_g -\rho_1 \left(\log\int_M \widehat h\,\re^{F+\wtilde G}\,dV_g + \log \int_M  \widehat h\,\re^{F-\wtilde G}\,dV_g\right) + C, 
\end{split}
\end{align}
where $C>0$ depends only on $h_1, h_2$, $\rho_1, \rho_2$ and $M$. Next, by H\"{o}lder's inequality one has
$$
	\int_M \widehat h\,\re^F\,dV_g = \int_M \widehat h\,\re^{\frac{F+\wtilde G}{2}} \re^{\frac{F-\wtilde G}{2}}\,dV_g \leq \left( \int_M \widehat h\,\re^{F+\wtilde G}\,dV_g  \right)^{\frac 12} \left( \int_M \widehat h\,\re^{F-\wtilde G}\,dV_g  \right)^{\frac 12},
$$
and thus, we deduce
$$
	2\log \int_M \widehat h\,\re^F\,dV_g \leq \log \int_M \widehat h\,\re^{F+\wtilde G}\,dV_g + \log \int_M \widehat h\,\re^{F-\wtilde G}\,dV_g.
$$
Using the latter estimate in \eqref{est} one gets
$$
	\wtilde J_\rho(F) \leq \int_M |\n F|^2 \,dV_g -2\rho_1 \log \int_M \widehat h\,\re^F\,dV_g+C.
$$
\end{proof}

With this in hand, we can construct the desired map as follows.

\smallskip

\begin{pro} \label{pro:test}
Suppose $\rho_1,\rho_2>8k\pi$ and let $K$ be a compact subset of $M\setminus\{p_1,\dots,p_N\}$, $K_k\subset M_k$ be the formal barycenters supported in $K$. For any $L>0$ sufficiently large there exists a map $\Phi: K_k\to \wtilde J_\rho^{-L}$.
\end{pro}

\begin{proof}
Fixing $L>0$ we set $\Phi: K_k\to \wtilde J_\rho^{-L}$ as
\be \label{phi}
	\Phi(\s)= \varphi_{\l,\s}-\ov{\varphi}_{\l,\s}, \quad \s\in K_k,
\ee
for a suitable $\l>0$ large enough, where $\varphi_{\l,\s}$ is given in \eqref{test}. Indeed, it is enough to prove that $\wtilde J_\rho(\Phi(\s))\leq-L$ uniformly in $K_k$. Without loss of generality, suppose that $\rho_1\leq\rho_2$. 
By first applying Lemma \ref{lem:est2} and then using the estimates of Lemma \ref{lem:test}, we deduce
\begin{align*}
	\wtilde J_\rho(\Phi(\s)) & \leq \int_M |\n \varphi_{\l,\s}|^2 \,dV_g -2\rho_1 \log \int_M \widehat h\,\re^{\varphi_{\l,\s}-\ov{\varphi}_{\l,\s}}\,dV_g +C\\
	&\\
	                         & \leq 4 (8k\pi-\rho_1+o_\l(1))\log\l \leq -L,
\end{align*}
for $\l>0$ sufficiently large, where in the last inequality we have used the fact that $\rho_1>8k\pi$.
\end{proof}

\medskip

We have now all the ingredients to prove the main existence results in Theorem~\ref{thm1}. We will treat the zero and positive genus case separately.

\medskip

\subsection{The zero genus case} Let us start with the first part of Theorem~\ref{thm1}, the case $M=\mathbb{S}^2$ with $\alpha_j=0$ for all $j$.

\medskip

\begin{proof}[Proof of \textbf{1.} in Theorem \ref{thm1}.]
We recall that we have $(\rho_1,\rho_2)\notin\L$ with
\begin{align*}
 \rho_1,\rho_2>8k\pi, \quad \frac{\rho_1+\rho_2}{2} < 8(k+1)\pi.
\end{align*}
The goal is to apply Lemma \ref{lem:def}. Now, by Proposition \ref{pro:proj}, there exists a projection $\Psi: \wtilde J_\rho^{-L}\to (\mathbb{S}^2)_k$, for some sufficiently large $L>0$.
Since we are in the regular case, Proposition \ref{pro:test} can be applied directly with $K=\mathbb{S}^2$, and thus, there exists a map $\Phi: (\mathbb{S}^2)_k\to \wtilde J_\rho^{-L}$, see \eqref{phi}. 
By construction of such maps and by the property \eqref{conv}, it is not difficult to show that the composition 
\begin{align*}
	&(\mathbb{S}^2)_k \stackrel{\Phi}{\longrightarrow} \wtilde J_\rho^{-L} \stackrel{\Psi}{\longrightarrow} (\mathbb{S}^2)_k \\
	&\s\mapsto (\varphi_{\l,\s}-\ov{\varphi}_{\l,\s}) \mapsto \frac{\re^{\varphi_{\l,\s}}}{\int_{\mathbb{S}^2} \re^{\varphi_{\l,\s}}} \simeq \s
\end{align*}
is homotopic to the identity on $(\mathbb{S}^2)_k$. Here, the homotopy being achieved by letting $\l\to+\infty$; see, for example, Proposition 4.4 in \cite{mal2}. Passing to the induced maps $\Phi^*,\Psi^*$ between homology groups, we derive $\Psi^*\circ\Phi^*=\mbox{Id}_{(\mathbb{S}^2)_k}^*$. 
In particular, the $q$-th homology groups of $(\mathbb{S}^2)_k$ are mapped injectively into the $q$-th homology groups of $\wtilde J_\rho^{-L}$, i.e.
$$
	H_q((\mathbb{S}^2)_k) \hookrightarrow H_q(\wtilde J_\rho^{-L}), \qquad q \ge 0.
$$  
Since it is well-known that $(\mathbb{S}^2)_k$ is not contractible (we refer the interested readers to Lemma 4.1. in \cite{mal2}), we conclude that $\wtilde J_\rho^{-L}$ is not contractible as well. On the other hand, $\wtilde J_\rho^{L}$ is contractible by Proposition \ref{pro:contr}. The existence of a solution to \eqref{system} is then guaranteed by Lemma \ref{lem:def}.
\end{proof}

\medskip

\subsection{The positive genus case} We next move to the second part of Theorem~\ref{thm1} so we consider a surface $M$ with genus $\textbf{g}>0$ and $\a_j\geq0$ for all $j$. To get multiplicity of solutions we will need weak Morse inequalities which we recall here. 

\medskip

For $q\in \N$ and a topological space $X$, we will denote by $H_q(X)$ its $q$-th homology group with
coefficient in $\Z$. For a subspace $A\subseteq X$, we write $H_q(X,A)$ for the $q$-th relative homology group of $(X,A)$. We will denote by $\wtilde H_q(X)$ the reduced $q$-th homology group, i.e. $H_0(X)\cong\wtilde H_0(X) \oplus \Z$ and $H_q(X)\cong\wtilde H_q(X)$ for all $q>0$. The $q$-th Betti number of $X$ will be indicated by $\b_q(X)= \mbox{rank } (H_q(X))$, while $\wtilde\b_q(X)$ will correspond to the rank of the reduced homology group.

\medskip

Letting $N$ be a Hilbert manifold, we recall that a function $f\in C^2(N,\R)$ is called a Morse function if all its critical points are non-degenerate. 
For $a<b$, we set
$$
	\begin{array}{c}
		C_q(a,b) = \# \bigr\{  \mbox{critical points of $f$ in $\{a \leq f \leq b\}$ with index $q$}  \bigr\}, \vspace{0.3cm} \\ 
		\b_q(a,b) = \mbox{rank } \bigr( H_q \bigr( \{ f \leq b \}, \{ f \leq a \} \bigr) \bigr).
	\end{array}
$$
For a proof of the following result, we refer, for example, to Theorem 4.3 in \cite{ch}.

\smallskip

\begin{pro} \label{pro:morse}
Let $N$ be a Hilbert manifold, and let $f\in C^2(N,\R)$ be a Morse function satisfying the Palais-Smale condition. Let $a<b$ be regular values of $f$. Then,
$$
	C_q(a,b) \geq \b_q(a,b), \qquad  q\ge 0.
$$
\end{pro}

\smallskip

We point out that the Palais-Smale condition is not required here; it can be replaced by suitable deformation lemmas for $f$ (see Lemma 3.2 and Theorem 3.2 in \cite{ch}), which in our case follow by the arguments in \cite{mal}. 
On the other hand, one can apply a transversality result from \cite{sa-te} (see, e.g., Theorem 1.5 in \cite{demar}) to conclude that, generically, all critical points of $\wtilde J_\rho$ are non-degenerate in the following sense.

\smallskip

\begin{lem} \label{lem:non-deg}
Suppose $\rho=(\rho_1,\rho_2)\notin \L$. Then, for $(g,h_1,h_2)$ in an open dense subset of $\M^2 \times C^{\,2}(M)\times C^{\,2}(M)$, $\wtilde J_\rho$ is a Morse function.
\end{lem}

\smallskip

We will say that the above property holds for a generic choice of $(g,h_1,h_2)$. We start by observing that by Proposition \ref{pro:contr} and by long exact sequence of the relative homology, one has
$$
	H_{q+1}\left( J_\rho^L, J_\rho^{-L} \right) \cong \wtilde H_q \left(J_\rho^{-L}\right), \quad q\geq 0.
$$
Therefore, Proposition \ref{pro:morse} and Lemma \ref{lem:non-deg} yield  the following result.

\smallskip

\begin{lem} \label{lem:low-sub}
Suppose $\rho=(\rho_1,\rho_2)\notin \L$. Then, there exists $L>0$ such that for a generic choice of $(g,h_1,h_2)$ we have
$$
\#\bigr\{ \mbox{solutions of \eqref{system}} \bigr\} \geq \wtilde \b_q\left(J_\rho^{-L}\right), \quad q\in\N.
$$
\end{lem}

\smallskip

We can now prove the main result.

\medskip

\begin{proof}[Proof of \textbf{2.} in Theorem \ref{thm1}.]
Recall that $(\rho_1,\rho_2)\notin\L$ with
\begin{align*}
 \rho_1,\rho_2>8k\pi, \quad \frac{\rho_1+\rho_2}{2} < 8(k+1)\pi.
\end{align*}
By Lemma \ref{lem:low-sub}, it remains to estimate the rank $\wtilde \b_q\left(J_\rho^{-L}\right)$ of the homology of low sublevel sets.
To avoid the effect of the singular points $p_j$, we follow the argument in \cite{BdMM}. We first recall that a bouquet $\mathcal B_m$ of $m$ circles is defined as $\mathcal B_m = \cup_{i=1}^m \mathcal S_i$, where $\mathcal S_i$ is homeomorphic to $\mathbb{S}^1$ and $\mathcal S_i \cap \mathcal S_j = \{p\}$. Since $M$ has positive genus $\textbf{g}>0$, it is easy to show (see the proof of Proposition 3.1 in \cite{BdMM}) that there exists a bouquet $\mathcal B_{\textbf{g}}\subseteq M$ of $\textbf{g}$ circles such that
\begin{itemize}
\item[(i)] $\mathcal B_{\textbf{g}}\cap\{p_1,\dots,p_N\}=\emptyset$,

\item[(ii)] there exists a global retraction $\Pi:M\to\mathcal B_{\textbf{g}}$.
\end{itemize}
Now, by Proposition \ref{pro:proj}, there exists a projection $\Psi: \wtilde J_\rho^{-L}\to M_k$, for some sufficiently large $L>0$. 
Therefore, we can define a map $\Psi_{\Pi}: \wtilde J_\rho^{-L}\to (\mathcal B_{\textbf{g}})_k$ through the composition 
$$
	\wtilde J_\rho^{-L} \stackrel{\Psi}{\longrightarrow} M_k \stackrel{\Pi}{\longrightarrow} (\mathcal B_{\textbf{g}})_k,
$$
where, with a little abuse of notation, we still denote by $\Pi$ the map between the formal barycenters induced by the retraction in the above point (ii). Next, by the point (i), we can apply Proposition \ref{pro:test} with $K=\mathcal B_{\textbf{g}}$ and get a map $\Phi: (\mathcal B_{\textbf{g}})_k\to \wtilde J_\rho^{-L}$, see \eqref{phi}. Arguing as in the proof of the first part of Theorem \ref{thm1}, it is easy to show that $\Psi_{\Pi}^*\circ\Phi^*=\mbox{Id}_{(\mathcal B_{\textbf{g}})_k}^*$, and in particular
$$
	H_q((\mathcal B_{\textbf{g}})_k) \hookrightarrow H_q(\wtilde J_\rho^{-L}).
$$  
The rank of $H_{2k-1}\bigl((\mathcal B_{\mathbf g})_k\bigr)$ was computed in Proposition 3.2 of \cite{BdMM}; more precisely,
$$
\beta_{2k-1}\bigl((\mathcal B_{\mathbf g})_k\bigr)=\binom{k+\mathbf g-1}{\mathbf g-1}.
$$
Thus, we deduce that 
$$
	\b_{2k-1}\left(J_\rho^{-L}\right)\geq \binom{k+\emph{\textbf{g}}-1}{\emph{\textbf{g}}-1}
$$
and the conclusion follows by Lemma \ref{lem:low-sub}.
\end{proof}

\noindent {\bf  Acknowledgements.} 
S.-H Moon was  supported   by the National Research Foundation of Korea(NRF) grant funded by the Korea government(MSIT) (No. RS-2022-NR072398, \ No. RS-2022-NR070702).

\medskip

\noindent {\bf Data Availability} Data sharing is not applicable to this article as no new data were created or analyzed in this study.

\medskip

\noindent {\bf Conflict of Interest} The authors declare that they have no conflict of interest.

\

\end{document}